%% file: GoeritzInv.tex
\documentclass[a4paper,12pt]{amsart}

\usepackage{amssymb}
\usepackage{amsmath}
\usepackage{graphicx}
\usepackage{extarrows}

\title{Goeritz Invariant of Torus Links}
\author{K. AHARA, and S. WATANABE}


\newtheorem{thm}{Theorem}[section]
\newtheorem{prop}[thm]{Proposition}
\newtheorem{lem}[thm]{Lemma}
\newtheorem{cor}[thm]{Corollary}

\theoremstyle{definition}
\newtheorem{definition}[thm]{Definition}

\begin{document}
\maketitle

{\small
{\bf abstrust.}\quad We obtain the full list of Goeritz invariants of all torus knots and links.
}

\input{./introduction.tex}

\input{./preliminary.tex}

\input{./oddcase.tex}

\input{./evencase.tex}

\input{./appendix.tex}

\par\bigskip

Kazushi Ahara,  Department of Frontier Media Science, Meiji University, 4-21-1 Nakano, Nakano-ku, Tokyo, 164-8525, Japan

kazuaha63@hotmail.co.jp

Shingo Watanabe, Department of Mathematics, Meiji University, 1-1-1 Higashi-Mita, Tama-ku, Kawasaki, Kanagawa, 214-8571, Japan

\end{document}

%% file: introduction.tex

\section{introduction}

Goeritz invariant $g(L)$ of a link (or a knot) $L$ is one of classical invariants of knots and it was defined by Goeritz \cite{goeritz} in 1930's.

To obtain Goeritz invariant $g(L)$, we calculate integral elementary divisors of a irreducible Goeritz matrix $G_1(D)$ determined by a link projection $D$.  (Usually we remove $1$'s from the set of integral elementary divisors.) 

Recently, Ikeda et al.\cite{ikeda} calculate Goeritz invariants of the torus links $T(3,q)$.  They show that there are infinite number of torus links with its Goeritz invariant of length more than one.  (In fact, $g(T(3,6k))=(0,0)$ and $g(T(3,6k+3))=(2,2)$.) 

In this article, we calculate Goeritz invariants of all torus links.
Goeritz invariant contains more information than the determinant and the nulity of links do, but it may have less information than Alexander ideals.  

We compare Goeritz invariant and Alexander polynomial.  If $V_L$ is a Seifert matrix of a link $L$,  Alexander polynomial is given by $\Delta_L(t)=\mathrm{det }(tV_L-\,{}^tV_L)$.
Alexander polynomials $\Delta_{T(p,q)}(t)$ for torus links $T(p,q)$ are computed by Murasugi (see Prop. \ref{prop:Alexander-toruslink}). 

It is known that the symmetrization $V_L+\,{}^tV_L$ has the same integral elementary divisors as those of irreducible Goerits matrix $G_1$.
So if the Goeritz invariants is of length $1$, it coincides to the determinant $\mathrm{det\,}(L)=|V_L+\,{}^tV_L|=\pm\Delta_L(-1)$.
The nulity $n(L)$ is the number of 0's in the Goeritz invariant.  It is also the order of zero at $t=-1$ of $\Delta_L(t)$.  
The coefficient of the top term at $t=-1$ of $\Delta_L(t)$ is the product of non-zero entries of the Goeritz invariant.

In this context, we know that if the number of non-zero entries of the Goeritz invariant is more than $1$, it has much information than Alexander polynomial. 

Here we state our result.
\begin{thm}\label{th:main}
Let $p,q$ be a pair of integers more than $1$.  Let $g(p,q)$ be the Goeritz invariants of the torus link (or knot) $T(p,q)$. Let $r$ be the GCD (greatest common divisor) of $(p,q)$ and $p=p'r, q=q'r$.  \par
(1)  If both of $p,q$ are odd, then $g(p,q)=(2,\cdots, 2)$ ($(r-1)$-times $2$.)   ( If $p,q$ are co-prime, then $g(p,q)=(1)$.)\par
(2) If $p$ is odd and $q$ is even, then $g(p,q)=(p',0,\cdots,0)$ ($(r-1)$-times $0$.)  (If $p,q$ are co-prime, then $g(p,q)=(p)$.  If $p'=1$ then $g(p,q)=(0,\cdots,0)$. ) \par
(3) If both of $p,q$ are even, then $g(p,q)=(2p'q',0,\cdots,0)$ ($(r-2)$-times $0$.)  (If $r=2$, then $g(p,q)=(2p'q')$.)  
\end{thm}

A Goeritz matrix $G(T(p,q))$ has a big size(, around $(pq/2)\times (pq/2)$ ), so we always consider how to downsize the matrix.  In the proof of our theorem, we only pursue how to transform matrices well(, indeed we only use elementary transformations of matrices).

The authors would like to indicate an observation on a proof of our theorem.  In the 'odd $p$' case, we use a kind of induction.  First we set $(m,n)=(p,q)$ with some parameters, and consider induction for $(m,n)$.  See for example, Lemma \ref{lem:C3}. 
There are two ways of descent for induction, one is $2m<n$ case ($(m,n) \mapsto (m,n-m)$) and the other is $2m>n$ case ($(m,n) \mapsto (2m-n,m)$).  There are inductive formula for signature of Torus knots by Murasugi, where there are two kinds of similar recursion formula for $2m<n$  and $2m>n$ cases.  Such induction reminds us of a certain relation between Goeritz invariant and the signature of torus knots.  

In the last section, we mention homology groups of the branched double covers of torus links, which are straightforward determined by Goeritz invariants. 

This paper is organized as follows.  
In Section 2 we prepare some notations and basic technical lemmas.  
In Section 3 we consider the case that $p$ is an odd number. 
In Section 4 we consider the case that $p,q$ are even numbers.
 When $p$ is even, the Goeritz matrix is much complicated than that of odd case, so we need very tricky way of calculation for even $p$.  
In Section 5 we remark some trivial corollaries about our result.
 
The authors would thank Prof. Tohru Ikeda and Prof. Jun Murakami for their kind advises.

%% file: preliminary.tex

\section{preliminary}

\subsection{Goeritz invariant}

In this subsection, we recall the definition of Goeritz matrix and Goeritz invariant of a general link.
Let $L$ be a link (or a knot) and let $D$ be the diagram of $L$.  We give a checkerboard coloring on $D$ and obtain the signatures of all crossings as follows.

\begin{figure}[hb]
\includegraphics[width=4cm]{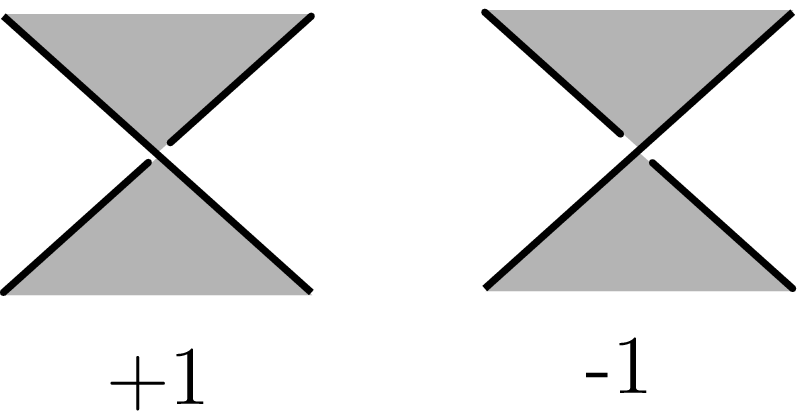}
\end{figure}

Let $R_0, R_1, \cdots, R_n$ be (black-)colored regions of the link projection. Using these regions, we define the Goeritz matrix as follows.  For a pair of different numbers $i,j$, let $g_{ij}=g_{ji}$ be the sum of signatures of all crossings where $R_i$ and $R_j$ intersect.  Let $g_{ii}$ be an integer satisfying an equation $\sum_{k=0}^ng_{ik}=0$.  Let the Goeritz matrix $G(D)$ be $\begin{pmatrix} g_{ij} \end{pmatrix}$, a matrix with elements $g_{ij}$ for $(i,j)$-entry. 

Let {\it an irreducible Goeritz matrix} $G_1(D)$ be a minor matrix of $G(D)$ that results from $G(D)$ by removing an arbitrary row and an arbitrary column.
By its definition, an irreducible Goeritz matrix is not well-defined according to choice of a removed row and a removed column.

We define the Goeritz invariant $g(K)$ by the sequence of integral elementary divisors (other than $1$s) of $G_1(D)$.  It is known that $g(K)$ does not depend on a link projection nor a checkerboard coloring nor choice of a removed row nor choice of a removed column.  See, for example, section 5.3 of \cite{kawauchi}.

We describe the definition of Goeritz invariant precisely.

\begin{definition}
Two integral matrices $M,N$ are called to be $\mathbb{Z}$- equivalent ($M \sim N$) if we transform $M$ into $N$ by the followings.\\(a)\ Interchange two rows (resp. two columns).\\
(b)\ Multiply a row (resp. a column) by $-1$.\\
(c)\ Add a row (resp. column) to another one multiplied by an integer.\\
(d)\ $M=(1)\oplus N$ or $N=(1)\oplus M$.
\end{definition}

The relations (a) or (b) or (c) are equivalent to the condition that there exists integral general linear (IGL) matrices $U_1,U_2\in GL(\mathbb{Z})$ such that $M=U_1NU_2$.

In the elementary divisor theory, it is known that any integral square matrix has a normalized form \[
\mathrm{diag}(1,\cdots, 1, d_1, d_2, \cdots, d_k)=(1)^{r\oplus}\oplus (d_1)\oplus(d_2)\oplus\cdots\oplus(d_k),
\]
and the next proposition follows.

\begin{prop}[Integral elementary divisors]\label{prop:ied}
Any integral square matrix $X$ is $\mathbb{Z}$-equivalent to $(d_1)\oplus(d_2)\oplus\cdots\oplus(d_k)$.  Here $d_i$ is a non-negative integer ($d_i\neq 1$) and there exists an integer $c_i$ such that $d_i=c_id_{i-1}$ for each $i$.  The sequence $(d_1,d_2,\cdots, d_k)$ is uniquely determined.
\end{prop}

\begin{definition}[Goeritz invariant]
For a given knot (or a link) $L$, we choose one projection $D$ and one checkerboard coloring on $D$.  We define Goeritz invariant $g(L)=(d_1,d_2,\cdots, d_k)$ by the unique sequence of the integral elementary divisors of a irreducible Goeritz matrix $G_1(D)$.
\end{definition}

\subsection{Notations of matrices}

In this subsection, we prepare some notations about matrices.  Let $q$ be a positive integer greater than $1$.  Let $M_q(\mathbb{Z})$ be a set of integral matrices of size $q\times q$.

Let $E=E_q$ be an identity matrix and $W=W_q$ be a cyclic permutation matrix given by 
\[
W=W_q=\begin{pmatrix} 0 &1 & &  \\  & 0 & \ddots & \\  & & \ddots & 1 \\ 1 & & & 0 \end{pmatrix} ,
\]
where these matrices are of size $q\times q$.
Remark that $W^q=E$ and $\text{det\,}W=(-1)^{q-1}$.  

Let $X=X_q$ be defined by $X_q=W_q+W_q^{-1}$. That is,

\[ X=X_q= \begin{pmatrix} 0 &1 & &  1 \\  1& 0 & \ddots & \\  & \ddots & \ddots & 1 \\ 1 & &1 & 0 \end{pmatrix}
\]

Let $N=N_q$ be a nil matrix given by 
\[  N=N_q=\begin{pmatrix} 
0 &1  & & & \\ 
 & 0 &\ddots & &\\ 
 &  &\ddots & \ddots &\\ 
 & &  & 0 & 1\\
 & & & & 0 \end{pmatrix}. 
 \]
Remark that $N^q=O$.  
 
\subsection{technical lemmas}

In this subsection,  we introduce $\mathbb{Z}[w,w^{-1}]$-equivalence and we show some technical lemmas.
 
 \begin{definition}
Let $M,N$ be matrices with entries of $\mathbb{Z}[w,w^{-1}]$.  $M$ and $N$ are called to be $\mathbb{Z}[w,w^{-1}]$-equivalent ($M\overset{\mathbb{Z}[w,w^{-1}]}{\sim}N$) if we transform $M$ into $N$ by the followings.\\
\quad (a)\ Interchange two rows (resp. two columns).\\
\quad (b)\ Multiply a row (resp. a column) by a unit in $\mathbb{Z}[w,w^{-1}]$.\\
\quad (c)\ Add a row (resp. column) to another one multiplied by a constant in $\mathbb{Z}[w,w^{-1}]$.\\
\quad (d)\ $M=(1)\oplus N$ or $N=(1)\oplus M$.
\end{definition}

We prepare two technical lemmas on $\mathbb{Z}[w,w^{-1}]$-equivalence for the proof of our theorem.

\begin{lem} \label{lem:key1}
Let $x$ be defined by $x=w+w^{-1}\in\mathbb{Z}[w,w^{-1}]$, and let $k\times k$ matrix $F_k(w)$ be
\[
F_k(w)=\begin{pmatrix}  -x& 1 &&&& \\
1 & -x & 1 &&& \\
& 1 & \ddots & \ddots && \\
&& \ddots & \ddots & 1 & \\
&&& 1 & -x & 1 \\
&&&& 1 & -x+1 \end{pmatrix} \in M_k(\mathbb{Z}[w,w^{-1}]).
\]
(1)\ $F_k(w) \overset{\mathbb{Z}[w,w^{-1}]}{\sim} \begin{pmatrix}\mathrm{det\,}F_k(w)\end{pmatrix}$.\\
(2)\ $\mathrm{det}F_k(w)= \displaystyle\sum_{i=-k}^{k}(-w)^i $.
\end{lem}

\begin{proof}
Interchanging rows of $F_k(w)$, we move the first row to the lowest row.  That is, we have
\[
F_k(w) \sim \begin{pmatrix}
 1 & -x & 1 &&& \\
& 1 & \ddots & \ddots && \\
&& \ddots & \ddots & 1 & \\
&&& 1 & -x & 1 \\
&&&& 1 & -x+1 \\
-x& 1 &0&\cdots&\cdots&0 
\end{pmatrix}.
\]
Remark that the part (from the first row to $(k-1)$-th row) of this matrix is upper triangle.  So, using Gaussian elimination, we have the following matrix for a polynomial $f(w)$.
\[
F_k(w) \sim \begin{pmatrix} 1 & & & * \\ & \ddots && \\ && 1 & \\ 0 &&& f(w)\end{pmatrix} \quad (f(w) \in \mathbb{Z}[w,w^{-1}]) 
\]
Remarking that the Gaussian elimination preserve the determinant of matrices and observing the last matrix, we show that the $(k,k)$-entry $f(w)$ satisfies $f(w)=\pm \mathrm{det }F_k(w)$, and the statement (1) follows.

(2)\ We show the formula by induction.  When $k=1$, the entry is $-x+1=-w^{-1}+1-w$.  When $k=2$, the left hand side is 
$$ \text{det}\begin{pmatrix}-w-w^{-1} & 1  \\ 1 & -w-w^{-1}+1
\end{pmatrix} =w^{-2}-w^{-1}+1-w+w^2$$
and the statement holds.

We assume that the formula (2) holds for $k=1,2,\cdots, h$.  Using Laplace expansion, we have 
\begin{align*}
 \mathrm{det}F_{h+1}(w)&=-x\ \mathrm{det}F_h(w)-\mathrm{det}F_{h-1}(w) \\
&=(w+w^{-1})\left(\displaystyle\sum_{i=-h}^{h}(-w)^i \right)
 -\left(\displaystyle\sum_{i=-(h-1)}^{h-1}(-w)^i \right) \\
&=\displaystyle\sum_{i=-(h+1)}^{h+1}(-w)^i ,
\end{align*}
and the proof completes.
\end{proof}

In the same way, it is easy to show the following lemma.

\begin{lem} \label{lem:key2}
Let a 
$k\times k$ matrix $\tilde F_k(w)$ be defined by  
\[
\tilde F_k(w)=\begin{pmatrix}  -x+1& 1 &&&& \\
1 & -x & 1 &&& \\
& 1 & \ddots & \ddots && \\
&& \ddots & \ddots & 1 & \\
&&& 1 & -x & 1 \\
&&&& 1 & -x+1 \end{pmatrix} \in M_k(\mathbb{Z}[w,w^{-1}]).
\]
\par
(1) $\tilde F_k (w)\overset{\mathbb{Z}[w,w^{-1}]}{\sim} \begin{pmatrix}\mathrm{det}\tilde F_k(w)\end{pmatrix}$ \par
(2) $\mathrm{det}\tilde F_k(w)= \mathrm{det}F_k(w)+\mathrm{det}F_{k-1}(w) $.
\end{lem}

%% file: oddcase.tex

\section{Proof of Theorem for odd $p$}

\subsection{Goeritz matrices of torus links}

In this subsection, we get a Goeritz matrix $G(T(p,q))$ of the torus knot (link) for a odd number $p$ ($3 \le p$) and an integer $q$ ($2 \le q$). 

\begin{prop}\label{prop:C1}
Suppose that $p$ is odd and $p=2k+1$. ($k=1,2,\cdots$) A Goeritz matrix $G(T(p,q))$ is given by the following.
\[
 G(T(p,q)) = \begin{pmatrix}
-q & {\boldsymbol{1}}&&&&& \\
^t \boldsymbol{1} & -X& E &&&& \\
& E & -X & E &&& \\
&& E & \ddots & \ddots && \\
&&& \ddots & \ddots & E & \\
&&&& E & -X & E \\
&&&&& E & -X+E 
\end{pmatrix}\in M_{kq+1}(\mathbb{Z})
\]
Here, $X=X_q=W_q+W_q^{-1}$, $E=E_q$, ${\boldsymbol{1}}=\begin{pmatrix} 1 & 1 & \cdots & 1 \end{pmatrix}\in M_{1,q}(\mathbb{Z})$
\end{prop}

\begin{proof}
First, we make an checkerboard coloring on the diagram of a torus link $T(p,q)$ as in the following figure.  We determine names of each regions of the diagram as follows.

\begin{figure}[hb]
\includegraphics[width=10cm]{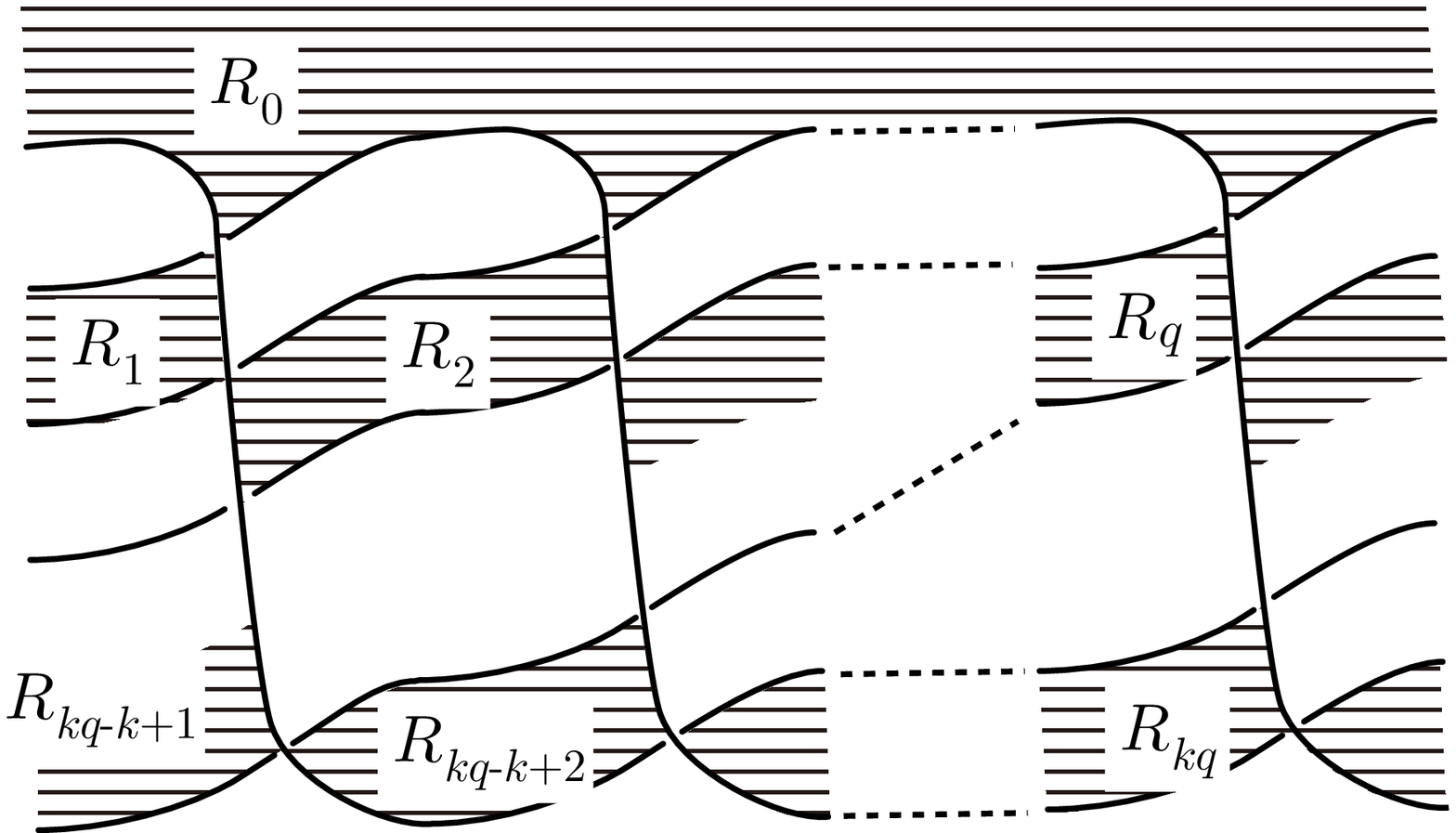}
\end{figure}
 
We obtain a Goeritz matrix immediately from this figure.
\end{proof}

Next, let $G_1(T(p,q))$ be a minor matrix that results from $G(T(p,q))$ by removing the first row and the first column.  We start to transform this matrix into smaller one. 

\begin{lem}@\label{lem:C2}\mbox{} Suppose that $p=2k+1$, $p\le q$, and \\
$G_1(T(p,q))=\begin{pmatrix}
 -X& E &&&& \\
 E & -X & E &&& \\
& E & \ddots & \ddots && \\
&& \ddots & \ddots & E & \\
&&& E & -X & E \\
&&&& E & -X+E 
\end{pmatrix}\in M_{kq}(\mathbb{Z}),$\\
then we obtain 
\[
G_1(T(p,q)) \sim  E_q+A_{p,q}+W_q^p.
\]
 Here $A_{p,q}$ is a $q\times q$ matrix given by \\ \hfil $ A_{p,q}=\left(\begin{array}{c|c}
& \\
O & O \\
& \\
\overbrace{\hphantom{-1\ 1\ -1\ \cdots \ -1}}^{p} & \overbrace{\hphantom{0\ 0\ \cdots \ 0}}^{q-p} \\
-1\ 1\ -1\ \cdots \ -1&0\ 0\ \cdots \ 0
\end{array} \right) $.
\end{lem}

\begin{proof}
Here we substitute $w=W$ for Lemma \ref{lem:key1}.  In the similar way of the proof of Lemma \ref{lem:key1} (1), we eliminate entries of $G_1(T(p,q))$ and obtain
\[ G_1(T(p,q))\sim \displaystyle\sum_{i=-k}^{k}(-W_q)^i.
\]

Next, we calculate 
$(E_q+N_q)W_q^k(\displaystyle\sum_{i=-k}^{k}(-W_q)^i)$.  Because $(E_q+N_q)W_q^k$ is unimodular, the resulting matrix is $\mathbb{Z}$-equivalent to $G_1(T(p,q))$.
\begin{align*}
&G_1(T(p,q)) \\
&\sim (E_q+N_q)W_q^k\left(\displaystyle\sum_{i=-k}^{k}(-W_q)^i\right) \\
&= E_q+N_q-(W_q+N_qW_q)+(W_q^2+N_qW_q^2)+\cdots \\
&\qquad +(-1)^{p-1}(W_q^{p-1}+N_qW_q^{p-1}) \\
&= E_q+ (N_q-W_q)(E_q-W_q+W_q^2-\cdots +W_q^{p-1}) + W_q^p
\end{align*}
Here, remark that $N_q-W_q=\begin{pmatrix}&\text{\LARGE $O$} \\ -1&\end{pmatrix}$, and we obtain
\[
 =E_q+A_{p,q}+W_q^p.
\]
\end{proof}

\subsection{Proof of \ref{th:main}(1)}\quad  
In this subsection, we show Theorem \ref{th:main} for odd $p,q$.  We suppose $p\le q$ without loss of generality.

First we define an $n \times n$ matrix $W(m,n,e_1,e_2,k,h)$ as follows.
\bigskip\\
$ W(m,n,e_1,e_2,k,h) = $\\
$ 
\hspace{11mm}\overbrace{\hphantom{{e_2}{\ddots}{e_2}\hspace{4\arraycolsep}}}^{m}
\hspace{2\arraycolsep}
\overbrace{\hphantom{{e_1}{\ddots}{\ddots}{e_1}\hspace{5\arraycolsep}}}^{n-m}
\\[2mm]
 = \left(\begin{array}{ccc|cccc}
1&&&e_1&&& \\
&\ddots &&&\ddots && \\
&& 1 &&&\ddots & \\
&&& \ddots &&& e_1 \\ \hline 
e_2 &&&& 1 && \\
& \ddots &&&& \ddots & \\
&& e_2 &&&& 1
\end{array} \right) 
\begin{array}{l}
\left. \vphantom{\begin{array}{l} \\ \\ \\ \\ \end{array}} \right\}\, n-m \\
\noalign{\smallskip}
\left. \vphantom{\begin{array}{l} \\ \\ \\ \end{array}} \right\}\, m 
\end{array}
$\\
$+
\left(\begin{array}{c|c}
& \\
O & O \\
& \\
\overbrace{\hphantom{-k\ k\ -k\ \cdots \ (-1)^{m}k}}^{m} & \overbrace{\hphantom{-h\ h\ -h\ \cdots \ (-1)^{n-m}h}}^{n-m} \\
-k\ k\ -k\ \cdots \ (-1)^mk&-h\ h\ -h\ \cdots \ (-1)^{n-m}h
\end{array} \right)
$\bigskip \\

Because of the assumption $p\le q$, $E_q+A_{p,q}+W_q^p$ is $W(p,q,1,1,1,0)$.  

For a matrix $W(m,n,e_1,e_2,k,h)$ of size $n\times n$, we consider the following two operations of elementary transformations.
\par\bigskip
(operation 1)  When $2m<n$, we add $j$-th row ($j=1,2,\cdots , m$) multiple by $-e_1$ to $(m+j)$-th row.

(operation 2)  When $2m<n$, we add $j$-th row ($j=1,2,\cdots , n-m$) multiple by $-e_1$ to $(m+j)$-th row.
\par\bigskip
After the operation 1, there remain diagonal elements in the $j$-th row ($j=1,2,\cdots,m$).  After the operation 2, there remain diagonal elements in the $j$-th row ($j=1,2,\cdots,n-m$).  Then Lemma \ref{lem:C3} follows .

\begin{lem}\label{lem:C3}
(1) When $2m<n$, using the operation 1, we show that $W(m, n, e_1, e_2, k, h)$ is $\mathbb{Z}$-equivalent to $W(m, n-m, e_1, -e_1e_2, h-e_1k, (-1)^mh)$.\par
(2) When $2m>n$, using the operation 2, we show that $W(m, n, e_1, e_2, \\ k, h)$ is $\mathbb{Z}$-equivalent to $W(2m-n, m, -e_1e_2, e_2, (-1)^{n-m}k, h-e_1k)$.
\end{lem}

From Lemma \ref{lem:C3}, we show the following key proposition.

\begin{prop}\label{lem:oddodd-lem}
Suppose that both of $p,q$ are odd and that $r$ is the GCD of $(p,q)$.  Then $W(p,q,1,1,1,0)$ is $\mathbb{Z}$-equivalent to $W(r,2r,1,-1,-1,0)$ or $W(r,2r,-1,1,0,1)$.
\end{prop}

\begin{proof}
First, set $m=p$ and $n=q$.  
When $n\neq 2m$, we can apply one of operations 1 and 2.  
After the operation, the GCD is preserved and the sum $m+n$ decreases.  
So we can apply operations until $(m,n)=(r,2r)$ within finite steps.

On the other hand,  transition of combinations of parities of $m,n$ and values  of $e_1,e_2,k,h$ are limited to the followings.
\begin{align*}
&W(\text{odd},\text{odd},1,1,1,0)\xLongleftrightarrow{\text{(1)}}  
W(\text{odd},\text{even},1,-1,-1,0) \xLongleftrightarrow {\text{(2)}}\\
&W(\text{even},\text{odd},1,-1,1,1) \xLongleftrightarrow{\text{(1)}}  
W(\text{even},\text{odd},1,1,0,1) \xLongleftrightarrow{\text{(2)}}  \\
&W(\text{odd},\text{even},-1,1,0,1) \xLongleftrightarrow{\text{(1)}}  
W(\text{odd},\text{odd},-1,1,1,-1) \xLongleftrightarrow{\text{(2)}} \\
&W(\text{odd},\text{odd},1,1,1,0)
\end{align*}
The GCD $r$ is odd number and the cases corresponding to$(m,n)=(r,2r)$ are $W(\text{odd},\text{even},1,-1,-1,0)$ or $W(\text{odd},\text{even},-1,1,0,1)$.  
This completes the proof.
\end{proof}

Using this proposition, we can show Theorem \ref{th:main} (1) immediately.

\begin{prop}
If $p,q$ are odd and $p \le q$, then 
$g(p,q)=(2,\cdots, 2)$ ($(r-1)$-times $2$.)   ( If $p,q$ are co-prime, then $g(p,q)=(1)$.)
\end{prop}

\begin{proof}
If the irreducible Goeritz matrix $G_1(T(p,q))$ is $\mathbb{Z}$-equivalent to $W(r,2r,1,-1,-1,0)$, then we obtain the result by the following calculation.  Here $A_r=A_{r,r}$.
\begin{align*}
&W(r, 2r, 1,-1,-1,0) \\
&=\begin{pmatrix}E_r&E_r \\ -E_r-A_r &E_r \end{pmatrix} \\
& \sim \begin{pmatrix}E_r&O \\ -E_r-A_r &2E_r+A_r \end{pmatrix} \\
& \sim 2E_r+A_r \\
&=\begin{cases}
\begin{pmatrix}2&&&& \\ &2&&& \\ &&\ddots && \\ &&& \ddots & \\ -1&1&\cdots&1&1 \end{pmatrix}=(2)^{(r-1)\oplus} & (r>1) \\
(1) & (r=1)
\end{cases}
\end{align*}
In the similar way, we can show it in the case of $W(r,2r,-1,+1,0,1)$.  This completes the proof.
\end{proof}

\subsection{Proof of Theorem \ref{th:main}(2)}\quad In this subsection we show our theorem for an odd $p$ and an even $q$.

\begin{prop}\label{lem:D1}
Suppose that $p$ is odd and $q$ is even. Let integers $s,t$ be determined by $p=qs+t$. ($0\le t<q$, $0 \le s $.)  Such integers are uniquely determined. \\
(1) The irreducible Goeritz matrix $G_1(T(p,q))$ of $T(p,q)$ is $\mathbb{Z}$-equivalent to $W(t,q,1,1,s+1,-s)$.\\
(2) There exists integers $k,h$ satisfying $|k-h|=p'=p/r$ such that $W(t,q,1,1,s+1,-s)$ is $\mathbb{Z}$-equivalent to $W(r,2r,1,1,k,h)$.
\end{prop}

\begin{proof}
(1) In the halfway of the proof of Lemma \ref{lem:C2} we have
\begin{align*}
G_1(T(p,q))\sim &E_q+ (N_q-W_q)(E_q-W_q+W_q^2-\cdots +W_q^{p-1}) + W_q^p \\
=& E_q+ W_q^p +\begin{pmatrix}&\text{\LARGE $O$} \\ -1&\end{pmatrix}(E_q-W_q+W_q^2-\cdots +W_q^{p-1}) .
\end{align*}
Because $q$ is even and $(W_q)^q=E$, (1) follows immediately.\\
(2) \ We consider in the same way in the former half of the proof of Proposition \ref{lem:oddodd-lem}.  First set $m=t, n=q$ and continue applying operations 1 and 2 until $(m,n)=(r,2r)$.

The transition of combinations of parities of $m,n$ and values of $e_1,e_2,k,h$ is limited as follows.
\begin{align}
& W(\text{odd},\text{even},1,1,k,h) \xLongrightarrow{\text{(1)}}W(\text{odd},\text{odd},1,-1,h-k,-h) \\
& W(\text{odd},\text{even},1,1,k,h) \xLongrightarrow{\text{(2)}}W(\text{even},\text{odd},-1,1,-k,h-k) \\
& W(\text{even},\text{odd},-1,1,k,h) \xLongrightarrow{\text{(1)}}W(\text{even},\text{odd},-1,1,h+k,h) \\
& W(\text{even},\text{odd},-1,1,k,h) \xLongrightarrow{\text{(2)}}W(\text{odd},\text{even},1,1,-k,h+k) \\
& W(\text{odd},\text{odd},1,-1,k,h) \xLongrightarrow{\text{(1)}}W(\text{odd},\text{even},1,1,h-k,-h) \\
& W(\text{odd},\text{odd},1,-1,k,h) \xLongrightarrow{\text{(2)}}W(\text{odd},\text{odd},1,-1,k,h-k) 
\end{align}
The case corresponding to $(m,n)=(r,2r)$ is only $W(\text{odd},\text{even},1,1,k,h)$.  (Here remark that
 $r$ is odd.)   So there exists integers $k,h$ such that $W(t,q,1,1,s+1,-s) \sim W(r,2r,1,1,k,h)$.

We have another observation from the above transition list.  When $m$ is odd $e_1=1$,  and when $m$ is even $e_1=-1$.  A formula $-e_1=(-1)^m$ follows.

Next we calculate the value of $|k-h|$.  
Consider the alternating sum of the $n$-th column of $W(m,n,e_1,e_2,k,h)$.  In fact we define an integer  $\ell(m,n,e_1,e_2,k,h)$ by $|km+(-1)^mh(n-m)|$.

Applying the operation 1, we have a map  $W(m,n,e_1,e_2,k,h)\mapsto W(m,n-m,e_1,-e_1e_2,h-e_1k,(-1)^mh)$.  
\begin{align*}
& \ell(m,n-m,e_1,-e_1e_2,h-e_1k,(-1)^mh)\\
&=|(h-e_1k)m+(-1)^m((-1)^mh)(n-2m)|\\
&=|-e_1km+h(n-m)|\\
&=|km+(-1)^mh(n-m)|=\ell(m,n,e_1,e_2,k,h).
\end{align*}
Therefore we show that 
\[ \ell(m,n,e_1,e_2,k,h)=\ell(m,n-m,e_1,-e_1e_2,h-e_1k,(-1)^mh) 
\]
holds.

Applying the operation 2, we have a map $W(m,n,e_1,e_2,k,h)\mapsto W(2m-n, m, -e_1e_2, e_2, (-1)^{n-m}k, h-e_1k)$.
\begin{align*}
& \ell(2m-n, m, -e_1e_2, e_2, (-1)^{n-m}k, h-e_1k)\\
&= |(-1)^{n-m}k(2m-n)+(-1)^{2m-n}(h-e_1k)(m-(2m-n))| \\
&= |(-1)^n\{ k((-1)^m(2m-n)+e_1(n-m))+h(n-m)\}| \\
&= |(-1)^n\{ (-1)^mkm+h(n-m)\}| \\
&= |(-1)^{n+m}(km+ (-1)^mh(n-m))|= |km+ (-1)^mh(n-m)|
\end{align*}
And we have 
\[
\ell(m,n,e_1,e_2,k,h)=\ell(2m-n, m, -e_1e_2, e_2, (-1)^{n-m}k, h-e_1k) .
\]

From the statement (1), we have $\ell(t,q,1,1,s+1,-s)=\ell(r,2r,1,1,k,h)$.  
Straightforward $\ell(r,2r,1,1,k,h)=|rh+(-1)^rrk|$ holds.  Since $r$ is odd, we have  $\ell(r,2r,1,1,k,h)=r|k-h|$.
On the other hand,
\[
\ell(t,q,1,1,s+1,s)=(s+1)t+(-1)^t(-s)(q-t)=t+sq=p
\]
is followed by $r|k-h|=p$ and we now get the result $|k-h|=p'$.
\end{proof}

Now we show (2) of the main theorem.

\begin{prop}[Theorem \ref{th:main} (2)]
Suppose that $p$ is odd and $q$ is even. Then
$g(p,q)=(p',0,\cdots,0)$ ($(r-1)$-times $0$.)  (If $p,q$ are co-prime then $g(p,q)=(p)$, if $p'=1$ then $g(p,q)=(0,\cdots,0)$) 'Å' 'éD
\end{prop}

\begin{proof}
From the above Proposition, the irreducible Goeritz matrix $G_1(T(p,q))$ is $\mathbb{Z}$-equivalent to $W(r,2r,1,1,k,h)$ and $|k-h|=p'$. If $A_r=A_{r,r}$ then 
\begin{align*}
&W(r,2r,+1,+1,h,k)\\
&=\begin{pmatrix}E_r&E_r \\ E_r+hA_r&E_r+kA_r\end{pmatrix} \\
& \sim (h-k)A_r \sim p'\begin{pmatrix} \begin{matrix} &&&& \\ &&O&&  \\ &&&& \end{matrix}\\ \begin{matrix} -1 & 1 & -1 & \cdots & 1 & -1\end{matrix} \end{pmatrix}
\end{align*}
Therefore we have $g(p,q)=(p',0,\cdots ,0)$.
Here there are $(r-1)$ times zeros.  If $p,q$ are co-prime, $r=1$ and there are no zero and $g(p,q)=(p)$. If $p'=1$(, that is, $p=r$) then $g(p,q)=(0,\cdots ,0)$ (, $1$ is removed).
\end{proof}

%% file: evencase.tex

\section{Proof of Theorem for even $p$}

\subsection{Goeritz matrices of torus links $T(p,q)$ for even $p$}

In this subsection, we get a Goeritz matrix $G(T(p,q))$ of the torus link for even numbers $p$ and $q$. 

\begin{prop}\label{prop:G1}
Assume that $p$ is an even positive integer. A Goeritz matrix  $G(T(p,q))$ is given by the following.
\[
 G(T(p,q))=
\begin{pmatrix}
 -X_q+E_q& E_q &&&& \\
 E_q & -X_q & E_q &&& \\
& E_q & \ddots & \ddots && \\
&& \ddots & \ddots & E_q & \\
&&& E_q & -X_q & E_q \\
&&&& E_q & -X_q+E_q 
\end{pmatrix}
 \label{eq:barG}
\]
Here there are $\frac{p}{2}$ blocks in rows and columns.
\end{prop}

See the following figure.

\begin{figure}[hb]
\includegraphics[width=10cm]{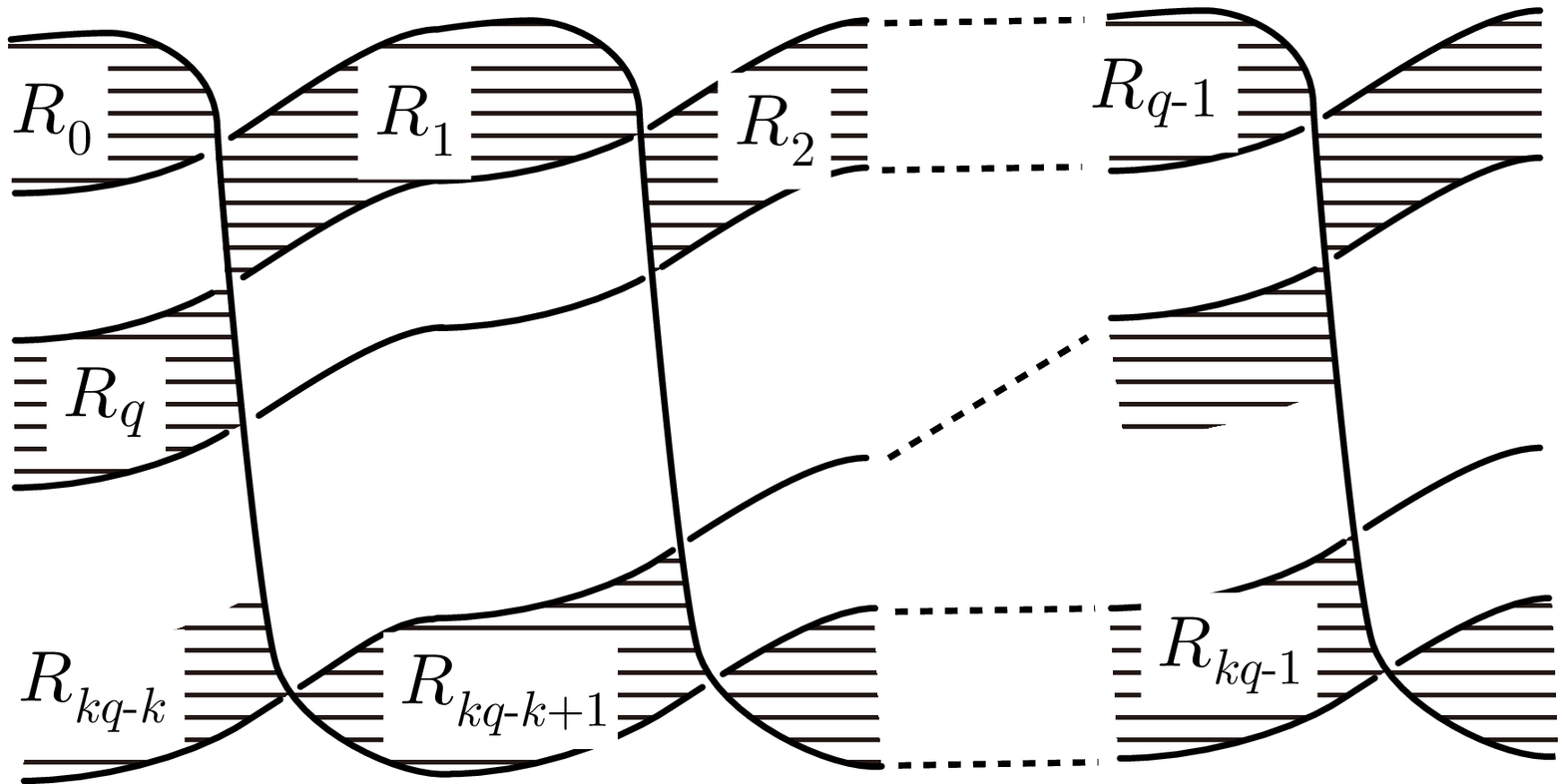}
\end{figure}

In this case, we first obtain the integral elementary divisors of \\$G(T(p,q))$.  The matrix $G(T(p,q))$ is sigular(, because the sum of elements in every row or in every column is zero), so the set of the integral elementary divisors contains at least one $0$.  Using this property we can calculate the set of integral elementary divisors of $G_1(T(p,q))$ from that of $G(T(p,q))$.

\begin{prop}\label{prop:G0}
Suppose that the integral elementary divisors of \\$G(T(p,q))$ are $d_1, d_2, \cdots, d_r, 0$.  Here $d_i$ is a non-negative  integer ($d_i \neq 1$) and there exists an ingeter $c_i$ such that $d_i=c_id_{i-1}$.  Then the integral elementary divisors of $G_1(T(p,q))$are $d_1,d_2, \cdots, d_r$ and hence the Goeritz invariant $g(p,q)$ is $(d_1, d_2, \cdots, d_r)$.
\end{prop}

\begin{proof}
First let $G$ be $G(T(p,q))$ for simplisity, and let $\ell$ be $\frac{pq}{2}$, the size of $G$.

From the assumption of the proposition, there exist IGL (integral general linear) matrices $U,U'$ such that 
\begin{equation}
UGU'=\mathrm{diag}(1,\cdots, 1,d_1, d_2, \cdots, d_r, 0). \tag{*}
\end{equation}
On the other hand, there exist $P, Q, S, P', Q', S'$ such that $U=PQS$, $U'=S'P'Q'$, where $P, P'$ are lower triangle IGL matrices, $Q,Q'$ are upper triangle IGL matrices, and $S,S'$ are permutation matrices.  (LU decomposition)

Let an $\ell\times\ell$ matrix $T$ be as follows.
\[
T=\begin{pmatrix} 1 &&& \\ & 1 && \\ && \ddots & \\ -1 & \cdots & -1&1 \end{pmatrix}
\]
By definition of a Goeritz matrix $G=\begin{pmatrix}g_{ij}\end{pmatrix}$, we have $\sum_{k=0}^{\ell-1}g_{ik}=\sum_{k=0}^{\ell-1}g_{ki}=0$ for any $i$.  Hence all entries of $\ell$-th row and those of $\ell$-th column of $T(S \,G \,S')\,{}^tT$ are zero.  So there exists a matrix $G'$ such that
\begin{equation}
T(S \,G \,S')\,{}^tT =
G' \oplus (0). \tag{**}
\end{equation}
Here $G'$ is a matrix that results from $G$ by changing rows and columns and removing one row and one column.  
Therefore $G'$ is $\mathbb{Z}$-equivalent to a irreducible Goeritz matrix $G_1(T(p,q))$.  ($G_1(T(p,q)) \sim G'$.)

Substituting the formula (*) to (**), we obtain the following.
\[
(PQT^{-1}) (G' \oplus (0) ) (\,^tT^{-1}P'Q') =\mathrm{diag}(1,\cdots, 1,d_1, d_2, \cdots, d_r, 0).
\]
Next, we set $PQT^{-1}={\hat P}{\hat Q}$, and $\,^tT^{-1}P'Q'={\hat P'}{\hat Q'}$,  where ${\hat P},{\hat P'}$ are lower triangle IGL matrices, and ${\hat Q},{\hat Q'}$ are upper triangle IGL matrices. 

There exists a matrix $G''$ such that
\[
{\hat Q}(G' \oplus (0)){\hat P'} =G'' \oplus (0) ,
\]
 and clearly $G'\sim G''$ holds(, because the $(\ell,\ell)$-minor of ${\hat Q}, {\hat P'}$ are also IGL matrices).  And,
\[
{\hat P}(G'' \oplus (0)){\hat Q'} =\mathrm{diag}(1,\cdots, 1,d_1, d_2, \cdots, d_r, 0)
\]
is followed by $G'' \sim \mathrm{diag}(1,\cdots, 1,d_1, d_2, \cdots, d_r)$.  (Because the $(\ell,\ell)$-minor of ${\hat P}, {\hat Q'}$ are also IGL matrices.)  
We have a conclusion that the integral elementary divisors of $G_1(T(p,q))$ are $d_1,d_2,\cdots,d_r$ and complete the proof of Proposition \ref{prop:G0}.
\end{proof}

In the sequel, we assume that both $p$ and $q$ are even numbers and that $p\leq q$.  Let $r$ be the GCD of $(p,q)$ and let $p',q'$ be integers such that $p=p'r, q=q'r$.  Remark that $r$ is an even number.  
The size of $G(T(p,q))$ in Proposition \ref{prop:G1} is $\frac{pq}{2}\times \frac{pq}{2}$.
We will downsize this matrix by $\mathbb{Z}$-equivalence with the following three steps. 

(step 1) Down to $q \times q$ matrix, (step 2) down to $p \times p$ matrix, and (step 3) down to $2r \times 2r$ matrix.

\subsection{Down to $q \times q$ matrix}

In this subsection, we downsize $G$ (in Proposition \ref{prop:G1}) into a $q \times q$ matrix by $\mathbb{Z}$-equivalence.  
  Using a similar way of proof of \ref{lem:key2}, we can show that $G(T(p,q))$ is $\mathbb{Z}$-equivalent to ${\tilde F}_k(W_q)$.  Let $G_2$ be ${\tilde F}_k(W_q)$.

\begin{align*}
&G(T(p,q)) \sim {\tilde F}_k(W_q)=F_k(W_q)+F_{k-1}(W_q)\\
 &= \displaystyle\sum_{i=-k}^{k}(-W_q)^i+\displaystyle\sum_{i=-(k-1)}^{k-1}(-W_q)^i\\
&= (-W_q)^{-k}+2\displaystyle\sum_{i=-(k-1)}^{k-1}(-W_q)^i+(-W_q)^k\\
&= (E_q-2W_q+W_q^2)(E_q+W_q^2+W_q^4+\cdots + W_q^{p-2})(-W_q)^{-k}\\
(&=:G_2)
\end{align*}

$G_2$ is already $q\times q$ matrix, but we furthermore transform $G_2$ and regulate the form.
Let $G_3$ be $(E_q+N_q)^2G_2(-W_q)^k$.  We calculate this matrix and obtain the follows.
\begin{align*}
&G_3=(E_q+N_q)^2G_2(-W_q)^p\\
&=
\begin{array}{@{}l@{}}
\hspace{6mm}
\overbrace{\hphantom{{E_2}{-E_2}{\ddots}{B_2-E_2}\hspace{8\arraycolsep}}}
^{p\  (k\text{ blocks})} \hspace{2\arraycolsep}
\overbrace{\hphantom{{-E_2}{-E_2}{-E_2}{\ddots}{-E_2}\hspace{6\arraycolsep}}}
^{q-p} \\
\left(
\begin{array}{cccc|ccccc}
 E_2 & -E_2 & &  &-E_2 & E_2 &  &  & O\\
  & E_2 & \ddots & &  & -E_2 & E_2 &&\\
&  & \ddots & \ddots &&& \ddots & \ddots & \\
&&  & \ddots & \ddots &&& \ddots & E_2 \\ 
 E_2 &&&&\ddots  & \ddots &&& -E_2 \\ \hline
 -E_2 & \ddots && & &\ddots & \ddots && \\
& \ddots & E_2 && &&\ddots & \ddots & \\
&& -E_2 & E_2 & &&&\ddots & -E_2  \\
 B_2 &\cdots &B_2 & B_2-E_2  & &&&&E_2  \\
 \end{array}
\right)
 \begin{array}{@{}l@{}}
 \left.\vphantom{\begin{array}{@{}c@{}}\\ \\ \\ \\ \\ \\ \\ \end{array}}\right\}\, \text{\scriptsize $q-p$} \\
 \noalign {\smallskip}
  \left.\vphantom{\begin{array}{@{}c@{}}\\ \\ \\ \\ \\ \end{array}}\right\}\, \text{\scriptsize $p$}
\end{array}
\end{array}
\end{align*}
Here $B_2$ is a $2\times 2$ matrix $\begin{pmatrix}-2 & 2 \\ -2 & 2 \end{pmatrix}$.   Since $(E_q+N_q)^2$, $(-W_q)^p$ are IGL matrices, we have $G_2\sim G_3$.

\subsection{Down to $p \times p$ matrix}

In the sequel, Let $B_{m,n}$ be defined by 
\[
B_{m,n}=\begin{pmatrix}-2&2&-2&2&\cdots \\ -2&2&-2&2&\cdots \\\cdots &&&&\end{pmatrix}\in M_{m,n}(\mathbb{Z}).
\]
We focus on the right-upper half of the matrix $G_3$.  A matrix 
\begin{align*}
&
\begin{array}{@{}l@{}}
\hspace{3.5mm}
\overbrace{\hphantom{{E_2}{-E_2}{\ddots}{-E_2}\hspace{5\arraycolsep}}}
^{p} \hspace{1.5\arraycolsep}
\overbrace{\hphantom{{-E_2}{-E_2}{\ddots}{-E_2}\hspace{5\arraycolsep}}}
^{q-p} \\
\left(
\begin{array}{cccccccc}
 E_2 & -E_2 & &  &-E_2 & E_2 &  &   O\\
& E_2 & \ddots & &  & \ddots & \ddots &\\
&& \ddots & \ddots &&& \ddots & E_2  \\
&&& \ddots & \ddots &&& -E_2 \\ 
&&&&\ddots  & \ddots && \\ 
&&&&&\ddots & \ddots & \\
&O&&&&&\ddots & -E_2  \\
&&&&&&&E_2  \\
\end{array}
\right)
\end{array}
\\
&=(E_q-N_q^2)(E_q-N_q^p).
\end{align*}
is an IGL matrix and its inverse matrix is  
\[
((E_q-N_q^2)(E_q-N_q^p))^{-1}=(E_q+N_q^p+N_q^{2p}+\cdots)(E_q+N_q^2+N_q^{4}+\cdots).
\]
We divide $G_3$ into three parts, that is,
\begin{align*} 
G_3 
&= (E_q-N_q^2)(E_q-N_q^p) \\
&+ 
\left(
\begin{array}{c|c}
O &
O \\
B_{2,p} &
O 
\end{array}
\right)+
\left(
\begin{array}{c|c}
\begin{array}{ccc} && \\ &O& \\ E_2&& \\ -E_2&\ddots& \\ &\ddots & E_2 \\ &&-E_2\end{array} &
\begin{array}{ccc} && \\ &O& \\ &&\end{array} 
\end{array}
\right)
\end{align*}

We calculate $((E_q-N_q^2)(E_q-N_q^p))^{-1}G_3$ one by one. 
\[ ((E_q-N_q^2)(E_q-N_q^p))^{-1}
\left(
\begin{array}{c|c}
O &
O \\
B_{2,p} &
O 
\end{array}
\right)
=\left(
\begin{array}{c|c}
 \vdots  & O\\
3B_{p,p} & O\\
2B_{p,p} & O \\
B_{p,p} &  O
\end{array}
\right),
\]
\[ ((E_q-N_q^2)(E_q-N_q^p))^{-1}
\begin{array}{@{}l@{}}
\hspace{3.5mm}
\overbrace{\hphantom{{-E_2}{-E_2}{\cdots}\hspace{7\arraycolsep}}}
^{p} \hspace{1.5\arraycolsep}
\overbrace{\hphantom{{O}\hspace{6\arraycolsep}}}
^{q-p} \\
\left(
\begin{array}{c|c}
\begin{array}{ccc} && \\ &O& \\ E_2&& \\ -E_2&\ddots& \\ &\ddots & E_2 \\ &&-E_2\end{array} &
\begin{array}{ccc} && \\ &O& \\ &&\end{array} 
\end{array}
\right)
\end{array}
\]
\[ 
=\left(
\begin{array}{c|c}
 \vdots & O \\
 -E_p & O \\
  -E_p & O 
\end{array}
\right).
\]
We sum them up and we know that there exists a matrix $G_4 \in M_p(\mathbb{Z})$ such that 
\[
((E_q-N_q^2)(E_q-N_q^p))^{-1}G_3
=
\begin{array}{@{}l@{}}
\hspace{3mm}
\overbrace{\hphantom{\hspace{4mm}}}
^{p} \hspace{1\arraycolsep}
\overbrace{\hphantom{O\hspace{2\arraycolsep}}}
^{q-p} \\
\begin{pmatrix} \ G_4 & O \\ * & E_{q-p} \end{pmatrix}
\begin{array}{@{}l@{}}
 \left.\vphantom{3.5mm}\right\}\, \text{\scriptsize $p$} \\
 \noalign {\smallskip}
  \left.\vphantom{3.5mm}\right\}\, \text{\scriptsize $q-p$}
\end{array}
\end{array}.
\]
Suppose that integers $m,\alpha$ are given by $q=mp+\alpha$(, $0\le \alpha < p$), then explicitly we have
\[
G_4=E_p+
\begin{pmatrix}(m+1)B_{\alpha,p} \\ mB_{p-\alpha,p}\end{pmatrix}-W_p^{-\alpha},
\]
and $G_3$ and $G_4$ are $\mathbb{Z}$-equivalent to each other.

\subsection{Down to $2r \times 2r$ matrix and proof of the theorem \ref{th:main}(3)}

We divide into two cases:  case 1:$\alpha=0$, and case 2: $\alpha>0$.  When $\alpha=0$, $E_p-W_p^{-\alpha}=O$ and 
\[
G_4=mB_{p,p}\sim (2m)\oplus (0)^{(p-1)\oplus}
\]
holds.  From Proposition \ref{prop:G0},  we obtain
\[
g(p,mp)=(2m,0,\cdots,0)\quad ((p-2)\text{-times } 0),
\]
In this case, $p=r, p'=1, q'=m$ is followed by Theorem \ref{th:main}(3) immediately.
\par\bigskip
Next consider the case $\alpha>0$.  Let an integer $\alpha'$ be defined by  $\alpha'=\frac{\alpha}{r}$.  Remark that $q'=mp'+\alpha'$.  We divide $G_4$ into blocks of $r\times r$ matrices.  Let $G'_4, G''_4$ be defined by  
\[
G'_4 =
\begin{array}[b]{@{}c@{}}
\overbrace{\hphantom{{E_r}{\ddots}{E_r}\hspace{6\arraycolsep}}}
^{p-\alpha} \hspace{2\arraycolsep}
\overbrace{\hphantom{{E_r}{\ddots}{E_r}\hspace{6\arraycolsep}}}
^{\alpha} \\
\left( \begin{array}{@{\,}ccc|ccc@{\,}}
E_r & & & -E_r & & \\
& \ddots & & & \ddots &  \\
& & \ddots & & & -E_r \\ 
-E_r & & & \ddots & & \\
 & \ddots & & & \ddots & \\
 & & -E_r &  & & E_r \end{array}
 \right)
 \end{array}
 \begin{array}{@{}l@{}}
 \left.\vphantom{\begin{array}{@{}c@{}}\\ \\ \\ \end{array}}\right\}\, \text{\scriptsize $\alpha$} \\
 \noalign {\smallskip}
  \left.\vphantom{\begin{array}{@{}c@{}}\\ \\ \\ \end{array}}\right\}\, \text{\scriptsize $p-\alpha$}
\end{array} 
\]
\[
G''_4= \begin{pmatrix}(m+1)A_{r}&\cdots & (m+1)A_{r}  \\  \vdots& \\ (m+1)A_{r}&\cdots & (m+1)A_{r} \\  mA_{r} & \cdots & mA_{r}  \\
\vdots & \\ mA_{r} & \cdots & mA_{r}  \end{pmatrix}
 \begin{array}{@{}l@{}}
 \left.\vphantom{\begin{array}{@{}c@{}}\\ \\ \\ \end{array}}\right\}\, \text{\scriptsize $\alpha$} \\
 \noalign {\smallskip}
  \left.\vphantom{\begin{array}{@{}c@{}}\\ \\ \\ \end{array}}\right\}\, \text{\scriptsize $p-\alpha$}
\end{array} ,
\]
then we have $G_4=G'_4+G''_4$.  Let $T$ and $T'$ be defined by 
\[
T=\begin{pmatrix} E_r & -E_r & \cdots & -E_r \\
 & E_r && \\
 && \ddots & \\
 &&& E_r \end{pmatrix}, T'=\begin{pmatrix} E_r & E_r & \cdots & E_r \\
 & E_r && \\
 && \ddots & \\
 &&& E_r \end{pmatrix} 
 \in M_p(\mathbb{Z}),
\]
then we have 
\[
T' G'_4T 
=  \left( \begin{array}{@{\,}ccc|cccc@{\,}}
O & & & O & O& \cdots & O\\
& E_r & & & -E_r & & \\
& & \ddots & & & \ddots  & \\
& & & \ddots & & & -E_r \\ 
-E_r & E_r &  \cdots & E_r & 2E_r  &  \cdots & E_r\\
 & \ddots & & & & \ddots & \\
 & & -E_r & & & & E_r \end{array}
 \right) .
\]
Since $p'=p/r$ and $\alpha'=\alpha/r$ are co-prime, after changing some rows together and changing some columns together, we have 
\[
T' G'_4T \sim  \left( \begin{array}{@{\,}cccccc@{\,}}
O & \cdots & \cdots & \cdots & O\\
-E_r & 2E_r & E_r& \cdots & E_r &  \\
&  -E_r & E_r & & &  \\ 
&  &  \ddots & \ddots  &   & \\
&  &  & -E_r & E_r \end{array}
\right) 
\]
We make a right multiplication by $\begin{pmatrix} E_r & O & \cdots  & O \\ O & E_r &  &O   \\ \vdots & \vdots & \ddots & \\ O& E_r & \cdots & E_r\end{pmatrix}$ and we have
 \[
\sim  \left( \begin{array}{@{\,}cccccc@{\,}}
O & \cdots & \cdots & \cdots & \cdots &O\\
-E_r & p'E_r & (p'-2)E_r& \cdots & 2E_r &  E_r \\
*& O & E_r & & & O\\ 
\vdots & \vdots &   & \ddots   &  & \\
\vdots& \vdots &   & & \ddots    & \\
*& O &  O &  && E_r \end{array}
\right) .
\]
We make the same elementary transformations on $G''_4$ as we did on $G'_4$,  we have 
\[
G''_4 \sim  \begin{pmatrix}
(mp'+\alpha' ) B_r & O& \cdots  & O \\
 mB_r & \vdots  &  & \vdots \\
 * & \vdots &  & \vdots\\
 \vdots  & \vdots & &\vdots  \\
 * & O &\cdots & O  \end{pmatrix}.
\]
We sum them up and obtain 
\begin{align*}
G_4&\sim \begin{pmatrix}
 (mp'+\alpha' ) B_r & O& \cdots & \cdots & O \\
 mB_r-E_r & p'E_r  & * & \cdots & * \\
 * & O & E_r && \\
 * & \vdots & &E_r & \\
 * & O & &&E_r  \end{pmatrix}\\
&\sim \begin{pmatrix}
 (mp'+\alpha' ) B_r & O \\
 mB_r-E_r & p'E_r   
\end{pmatrix}
\sim \begin{pmatrix}
 (mp'+\alpha' ) B_r & p'(mp'+\alpha' ) B_r \\
 -E_r &    O
\end{pmatrix}\\
&\sim  p'(mp'+\alpha' ) B_r\sim (2p'q' )\oplus (0)^{(r-1)\oplus}.
\end{align*}
From Proposition \ref{prop:G0},
\[
g(p,q)=(2p'q',0,\cdots,0 ) \quad ((r-2) \text{- times } 0)
\]
and we complete the proof of Theorem \ref{th:main} (3). 

%% file: appendix.tex

\section{Relation with Alexander polynomials}

In this section we mention two topics about relation between Goeritz invariants, double branched cover of $S^3$, and Alexander polynomials of torus links. 

Kawauchi's book \cite{kawauchi} says that Georitz invariant  $(d_1,\cdots, d_k)$ and integral elementary divisors of the symmetrized matrix $V+\,^tV$ of the Seifert matrix $V$ coincide. (See Prop 8.2.2.jThese integral elementary divisors give the first homology of a double branched cover of the link.  That is, if $M_L$ is a double branched cover of $S^3$ along the link $L$, then 
\[
H_1(M_L) \equiv \mathbb{Z}_{d_1}\oplus \cdots \oplus \mathbb{Z}_{d_k}
\]
holds.  The following corollary follows this fact.

\begin{cor} Let $M=M_{T(p,q)}$ be a double branched cover of torus link $T(p,q)$.  \par
(1)  If both of $p,q$ are odd, then $H_1(M)\cong\mathbb{Z}_2^{(r-1)\oplus}$   ( If $p,q$ are co-prime, then $H_1(M)=O$.)\par
(2) If $p$ is odd and $q$ is even, then $H_1(M)\cong \mathbb{Z}_{p'}\oplus \mathbb{Z}^{(r-1)\oplus}$   (If $p,q$ are co-prime, then $H_1(M)\cong \mathbb{Z}_{p}$.  If $p'=1$ then $H_1(M)\cong  \mathbb{Z}^{(r-1)\oplus}$. ) \par
(3) If both of $p,q$ are even, then $H_1(M)\cong \mathbb{Z}_{2p'q'}\oplus \mathbb{Z}^{(r-2)\oplus}$   (If $r=2$, then $H_1(M)\cong \mathbb{Z}_{2p'q'}$.)  
\end{cor}

The second topic is on Alexander polynomial.  It is given by  $\Delta_L(t)=\det(tV-\,^tV)$.  Murasugi \cite{murasugi} shows the following formulas on torus links.

\begin{prop} \label{prop:Alexander-toruslink}
Alexander polynomial $\Delta(t)$ of the torus link $T(p,q)$ is given by follows.\par
(1)  If $p,q$ are co-prime, then
\[
\Delta(t)=t^{-\frac{(p-1)(q-1)}{2}}\dfrac{(1-t)(1-t^{pq})}{(1-t^p)(1-t^q)}
\]
(2) If the GCD $r$ of $(p,q)$ is more than 1, then
\[
\Delta(t)=t^{-\frac{(p-1)(q-1)}{2}}\dfrac{(1-t)(1-t^{pq/r})^r}{(1-t^p)(1-t^q)}.
\]
\end{prop}

On the above formulas, we define the zero-order $k$ and the top-coefficient $a_k$ by 
\[
t^{\frac{(p-1)(q-1)}{2}}\Delta(t)=a_k(t+1)^k+a_{k+1}(t+1)^{k+1}+\cdots.
\]
It is easy to show that (1) $k$ is nulity and that (2) $|a_k|$ equals to the product of non-zero elementary divisors of $\Delta(-1)$.  This is a small corollary of our theorem.

\begin{cor} Let $r$ be the GCD of $(p,q)$.\par
(1)  If both of $p,q$ are odd, then $k=1$ and $a_k=2^{r-1}$. \par
(2) If $p$ is odd and $q$ is even, then $k=r-1$ and
$a_k=p'=p/r$. \par
(3) If both of $p,q$ are even, then $k=r-2$ and $a_k=2p'q'=2pq/r^2$.  
\end{cor}

%% file: GoeritzInv.bbl
\nocite{*}
\begin{thebibliography}{9}
\bibitem[1]{goeritz}
L.\, Goeritz, {\it Knoten und quadratische Formen}, Math. Annalen, {\bf 103}, (1930), 647-654

\bibitem[2]{ikeda}
K. Ikeda, T. Ikeda, T. Kawakami, H. Sugimoto, T. Sugiura, J. Yagi, and S. Yamanaka, {\it Goeritz Invariants of Two-bridge Links and Totus Links}, Kochi J. Math., {\bf 5}, (2010), 163-172

\bibitem[3]{kawauchi}
A. Kawauchi, {\it Lecture of Knot Theory (in Japanese)}, (2007), Kyoritsu Shuppan

\bibitem[4]{murasugi}
K. Murasugi, {\it Knot theory and its application (in Japanese)}, (1993/2012), Nippon Hyouron-sha
\end{thebibliography}
